\documentclass[twoside,12pt]{article}
\usepackage{graphics}
\usepackage{graphicx}

\usepackage{amssymb,amsfonts}
\usepackage{amsmath}
\usepackage{amsthm}

\usepackage[makeroom]{cancel}

\graphicspath{{Fig/}}

\def\binomh#1#2{ \scalebox{.3}[1.2]{\textbf{)}}{\genfrac{}{}{0pt}{}{#1}{#2}}\scalebox{.3}[1.2]{\textbf{(}} }

\def\hpT45{${\cal HPT}_{\!4,5}$}
\def\hpTT46{${\cal HPT}_{\!4,6}$}
\def\hpt4q{${\cal HPT}_{\!4,q}$}

\newtheorem{theorem}{Theorem}
\newtheorem{lemma}{Lemma}[section]

\newtheorem{remark}{Remark}

\title{\bf Fibonacci words in hyperbolic Pascal triangles 
}

\author{L\'aszl\'o N\'emeth\footnote{University of Sopron,  Institute of Mathematics, Hungary; \textit{nemeth.laszlo@uni-sopron.hu}}}
\date{}

\date{}

\begin{document}

\maketitle

\begin{abstract}
	The hyperbolic Pascal triangle \hpt4q  $(q\ge5)$ is a new mathematical construction, which is a geometrical generalization of Pascal's arithmetical triangle. In the present study we show that a natural pattern of rows of \hpT45 is almost the same as the sequence consisting of every second term of the well-known Fibonacci words. Further, we give a generalization of the Fibonacci words using the hyperbolic Pascal triangles. The geometrical properties of a \hpt4q imply a graph structure between the finite Fibonacci words.\\[1mm]
	{\em Key Words: Hyperbolic Pascal triangle, Fibonacci word.}\\
	{\em MSC code:  05B30, 11B39}  
 
  The final publication is available at Acta Univ. Sapientiae, Mathematica, 2017 (www.acta.sapientia.ro/acta-math/matematica-main.htm).    
\end{abstract}


\section{Introduction}\label{sec:introduction}

The hyperbolic Pascal triangle \hpt4q  $(q\ge5)$ is a new mathematical construction, which is a geometrical generalization of Pascal's arithmetical triangle \cite{BNSz}.  In the present article we discuss the properties of the patterns of the rows of \hpt4q, which patterns give a new kind of generalizations of the well-known Fibonacci words. Our aim is to show the connection between the Fibonacci words and the hyperbolic Pascal triangles.   

After a short introduction of the hyperbolic Pascal triangles and the finite Fibonacci words we define a new family of Fibonacci words and we present the relations between the hyperbolic Pascal triangles and the newly generalized Fibonacci words. Their connections will be illustrated by figures for better comprehension. 
As the hyperbolic Pascal triangles are based on the hyperbolic regular lattices, their geometrical properties provide a graph structure between the  generalized finite Fibonacci words. {The extension of this connection could provide a new family of binary words.}

\subsection{Hyperbolic Pascal triangles} 

In the hyperbolic plane there are infinite types of regular mosaics (or regular lattices), that are denoted by the Schl\"afli symbol $\{p,q\}$, where $(p-2)(q-2)>4$. Each regular mosaic induces a so-called hyperbolic Pascal triangle, 
following and generalizing the connection between  classical Pascal's triangle and the Euclidean regular square mosaic $\{4,4\}$
(for more details see \cite{BNSz,NSz_alt,NSz2}). 

The hyperbolic Pascal triangle \hpt4q based on the mosaic $\{p,q\}$ can be depicted as a digraph, where the vertices and the edges are the vertices and the edges of a well-defined part of the lattice $\{p,q\}$, respectively. Further, the vertices possess a value each giving the number of the different shortest paths from the base vertex. Figure~\ref{fig:Pascal_46_layer5} illustrates the hyperbolic Pascal triangle when $\{p,q\}=\{4,6\}$. 
Generally, for a $\{4,q\}$ configuration  the base vertex has two edges, the leftmost and the rightmost vertices have three, the others have $q$ edges. The square shaped cells surrounded by appropriate edges  correspond to the regular squares in the mosaic.
Apart from the winger elements, certain vertices (called ``Type A'' for convenience) have two ascendants and $q-2$ descendants, the others (``Type B'') have one ascendant and $q-1$ descendants. In the figures of the present study we denote the  type $A$ vertices by red circles and the  type $B$ vertices by cyan diamonds, while the wingers by white diamonds. The vertices which are $n$-edge-long far from the base vertex are in row $n$.

\begin{figure}[h!]
	\centering
	\includegraphics[width=0.95\linewidth]{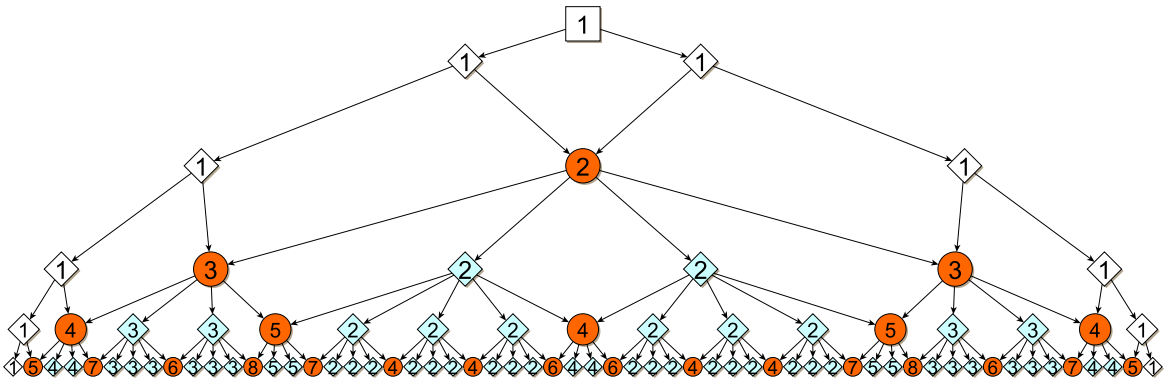}
	\caption{Hyperbolic Pascal triangle linked to $\{4,6\}$ up to row 5}
	\label{fig:Pascal_46_layer5}
\end{figure}

The general method of deriving the triangle is the following: going along the vertices of the $j^{th}$ row, according to the type of the elements (winger, $A$, $B$), we draw the appropriate number of edges downwards (2, $q-2$, $q-1$, respectively). Neighbour edges of two neighbour vertices of the $j^{th}$ row meet in the $(j+1)^{th}$ row, constructing a type $A$  vertex. The other descendants of row $j$ are type $B$ in row $j+1$. Figure~\ref{fig:Pascal growing_4q2} also shows a growing algorithm of the different types except the leftmost items, that are always types $B$ and $A$. (Compare Figure \ref{fig:Pascal growing_4q2} with Figures \ref{fig:Pascal_46_layer5}  and \ref{fig:Pascal_layer5_AB}.)

In the sequel, $\binomh{n}{k}$ denotes the $k^\text{th}$ element in row $n$, whose value is either the sum of the values of its two ascendants or the value of its unique ascendant. We note, that the hyperbolic Pascal triangle has the property of vertical symmetry. 

In the following we generalize the Fibonacci word in a new (but not brand new) way and show that this generalization is the same as the patterns of nodes types $A$ and $B$ in rows of \hpt4q.

\begin{figure}[h!]
	\centering
	\includegraphics[width=8cm]{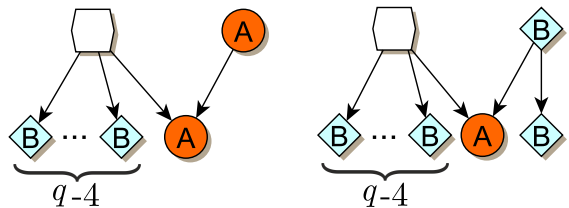}
	\caption{Growing method in Pascal triangles (except for the two leftmost items)}
	\label{fig:Pascal growing_4q2}
\end{figure}

\subsection{Fibonacci words}

The most familiar and the most studied binary word in mathematics is the Fibonacci word. The finite Fibonacci words,  $f_i$, are defined by the elements of the recurrence sequence $\{f_i\}_{i=0}^{\infty}$ over $\{0,1\}$ defined as follows
$$
f_0=1, \qquad f_1= 0, \qquad f_i=f_{i-1}f_{i-2},\qquad  (i\geq 2).
$$
It is clear, that $|f_i|= F_{i+1}$, where $F_i$ is the $i$-th Fibonacci number defined by the recurrence relation $F_i=F_{i-1}+F_{i-2}$ $(i\geq 2)$, with initial values $F_0=0$, $F_1=1$. The infinite Fibonacci word is $\boldsymbol{f}=\lim_{i\to \infty}f_i$. Table~\ref{tab:fibon_word} shows the first few Fibonacci words. 
It is also well-known that the Fibonacci morphism  ($\sigma$: $\{0,1\}\rightarrow\{0,1\}*$, $0\rightarrow01$, $1\rightarrow0$) acts between two consecutive finite Fibonacci words. For some newest properties (and further references) of Fibonacci words see \cite{Eds,Fici,lothaire,Rami2013,Rami}.

\begin{table}[htb]
	\centering 
	\begin{tabular}{rcl}
		$f_0$  & =& 1 \\
		$f_1$  & =& 0 \\
		$f_2$  & =& 01 \\
		$f_3$  & =& 010 \\
		$f_4$  & =& 01001 \\
		$f_5$  & =& 01001010 \\
		$f_6$  & =& 0100101001001 \\
		$f_7$  & =& 010010100100101001010 
	\end{tabular}
	\caption{The first Fibonacci words}
	\label{tab:fibon_word}
\end{table}

\section{$\{4,q\}$-Fibonacci words}

There are some generalizations of Fibonacci words, one of them is the biperiodic Fibonacci word \cite{Eds,Rami}. 
For any two positive integers $a$ and $b$, the biperiodic finite Fibonacci words sequence, say $\{\widehat{f}_i\}_{i=0}^{\infty}$, is defined recursively by
$$
\widehat{f}_{0}=1, \ 
\widehat{f}_{1}=0, \ 
\widehat{f}_{2}=0^{a-1}1=00\ldots 01,$$ and  
\begin{equation*}
\widehat{f}_{i}=
\begin{cases} 
\widehat{f}_{i-1}^{a}\widehat{f}_{i-2}, & \mbox{if } i \mbox{ is even}; \\ 
\widehat{f}_{i-1}^{b}\widehat{f}_{i-2}, & \mbox{if } i \mbox{ is odd}; \end{cases}\quad  (i\geq3).  
\end{equation*}
It has been proved \cite{Rami}, that if $i\geq1$ then $|\widehat{f}_i|= F_i^{(a,b)}$, where
for any two positive integers $a$ and $b$, the biperiodic Fibonacci sequence $\{F_i^{(a,b)}\}_{i=0}^{\infty}$ is defined recursively by
\begin{equation}\label{eq:bi-fib}
F_{0}^{(a,b)}=0, \ 
F_{1}^{(a,b)}=1, \ 
F_{i}^{(a,b)}=
\begin{cases} aF_{i-1}^{(a,b)}+F_{i-2}^{(a,b)}, &\mbox{if } i \mbox{ is even}; \\ 
bF_{i-1}^{(a,b)}+F_{i-2}^{(a,b)}, &\mbox{if } i \mbox{ is odd}; \end{cases}\quad (i\geq2).  
\end{equation}
The first few terms are $0$, $1$, $a$, $ab+1$, $a^2b+2a$, $a^2b^2+3ab+1$, $a^3b^2+4a^2b+3a$, $a^3b^3+5a^2b^2+6ab+1$.
When $a=b=k$, this generalization gives the $k$-Fibonacci numbers and in the case  $a=b=1$, we recover the original Fibonacci numbers \cite{Eds,Rami}. 

Now let us define  the finite  $\{4,q\}$-Fibonacci words sequence $\{f_i^{[4,q]}\}_{i=0}^{\infty}$, shortly $\{f_i^{[q]}\}_{i=0}^{\infty}$, where $q\geq5$, a new family of generalized Fibonacci words, and 
\begin{equation}\label{eq:4dFib}
f_{0}^{[q]}=1, \ 
f_{1}^{[q]}=0, \ 
f_{i}^{[q]}=
\begin{cases} {\left( f_{i-1}^{[q]}\right) }^{q-4} f_{i-2}^{[q]}, &\mbox{if } i \mbox{ is even}; \\ 
f_{i-1}^{[q]} f_{i-2}^{[q]}, &\mbox{if } i \mbox{ is odd}; \end{cases}\quad (i\geq2).  
\end{equation}
These new $\{4,q\}$-Fibonacci words are almost the same as the biperiodic Pascal words, $\widehat{f}_{i}$, if $a=1$ and $b=q-4$. As the definitions for the second items vary,  the  odd and even situations are reversing. 
If $q=5$, then $\{4,q\}$-Fibonacci words coincide with the classical Fibonacci words. (In Table~\ref{tab:fibon_wordq6} we list the first few $\{4,6\}$-Fibonacci words.)
The infinite $\{4,q\}$-Fibonacci word is defined as $\boldsymbol{f}^{[q]}=\lim_{i\to \infty}f_{i}^{[q]}$ and $\boldsymbol{f}= \boldsymbol{f}^{[5]}$ (see Table~\ref{tab:inffibon_word}). 

\begin{table}[htb]
	\centering 
	\begin{tabular}{rcll}
		$f_{0}^{[6]}$  & =& 1  \\
		$f_{1}^{[6]}$  & =& 0 \\
		$f_{2}^{[6]}$  & =& 001 \\
		$f_{3}^{[6]}$ & =& 0010 \\
		$f_{4}^{[6]}$  & =& 00100010001  \\
		$f_{5}^{[6]}$  & =& 001000100010010 \\
		$f_{6}^{[6]}$  & =& 00100010001001000100010001001000100010001 
	\end{tabular}
	\caption{The first few $\{4,6\}$-Fibonacci words}
	\label{tab:fibon_wordq6}
\end{table}

\begin{table}[htb]
	\centering 
	\begin{eqnarray*}	
		\boldsymbol{f}^{[5]}&=&01001010010010100101001001010010010100101001001010010100\ldots\\
		\boldsymbol{f}^{[6]}&=&00100010001001000100010001001000100010001001000100010010\ldots\\
		\boldsymbol{f}^{[7]}&=&00010000100001000010001000010000100001000010001000010000\ldots\\
		\boldsymbol{f}^{[8]}&=&00001000001000001000001000001000010000010000010000010000\ldots
	\end{eqnarray*}
	\caption{Some infinite $\{4,q\}$-Fibonacci words}
	\label{tab:inffibon_word}
\end{table}

In case of the extension of definition \eqref{eq:4dFib} to $q=4$, the  $f_{2k}^{[4]}=1$,  $f_{2k+1}^{[4]}=1\ldots 10$ (the number of $1$'s is $k$) for any $k\geq1$ and there is no limit of $f_i^{[4]}$ if  $i\to \infty$. Therefore, we investigate the $\{4,q\}$-Fibonacci words, when $q\geq5$. 

\medskip

Let $\sigma^{[q]}$ be the $\{4,q\}$-Fibonacci morphism defined by \begin{equation}
\{0,1\}\rightarrow\{0,1\}*,\quad 0\rightarrow0^{q-4}10, \  1\rightarrow0^{q-4}1, \label{eq:def4q}
\end{equation}
where $q\geq5$. 
\begin{theorem}\label{th:morf}
	The $\{4,q\}$-Fibonacci morphism, $\sigma^{[q]}$, acts between every second words of $\{4,q\}$-Fibonacci words, so that 
	\begin{equation}
	\sigma^{[q]}( f_{i-2}^{[q]}) =f_{i}^{[q]},\qquad (i\geq2). 
	\end{equation}
\end{theorem}
\begin{proof} 
	We prove the assertion by induction on $i$. The statement is clearly true for $i=2,3$.
	Now we assume, that the result holds for any $j$, when $4\leq j< i$. Let $i$ be first even. Then 
	\begin{eqnarray*}
		\sigma^{[q]}( f_{i-2}^{[q]})
		&=&\sigma^{[q]}\left( {\left( f_{i-3}^{[q]}\right) }^{q-4} f_{i-4}^{[q]}\right) 
		= {\left(\sigma^{[q]} (f_{i-3}^{[q]}) \right)}^{q-4} \sigma^{[q]} (f_{i-4}^{[q]})\\
		&=& {\left( f_{i-1}^{[q]}\right) }^{q-4} f_{i-2}^{[q]}=f_{i}^{[q]}
	\end{eqnarray*}
	If $i$ is odd the proof is similar, 
	$\sigma^{[q]}( f_{i-2}^{[q]})
	=\sigma^{[q]}\left( { f_{i-3}^{[q]} } f_{i-4}^{[q]}\right) 
	= \cdots =f_{i}^{[q]}.
	$
\end{proof} 

\begin{remark}
	$\sigma^{[5]}=\sigma^2$ and $\sigma^2(f_i)=f_{i+2}$.
\end{remark}

\section{Connection between \hpt4q and $\{4,q\}$-Fibonacci \newline words}

We consider again the hyperbolic Pascal triangle \hpt4q. Let us denote the left and right nodes '1'  by type $B$ (compare Figures \ref{fig:Pascal_46_layer5} and \ref{fig:Pascal_46_AB5}). Let $a_n$ and $b_n$ be the number of vertices of type $A$ and $B$ in row $n$, respectively. Further let 
\begin{equation}\label{sn}
s_n=a_n+b_n, 
\end{equation}
that gives the total number of the vertices in row $n\ge0$.  Then the ternary homogeneous recurrence relation
\begin{equation}\label{recur1}
s_n=(q-1)s_{n-1}-(q-1)s_{n-2}+s_{n-3} \qquad (n\ge4)
\end{equation}
holds with initial values 
$s_0=1$, $s_1=2$, $s_2=3$, $s_3=q$ (recall, that $q\ge5$). For the explicit form see \cite{BNSz}.

\begin{lemma}\label{lem:sn}
	If $n\geq1$, then
	\begin{equation}\label{eq:sn}
	s_n=u_n+2,
	\end{equation}
	where $u_1=0$, $u_2=1$ and  $u_n=(q-2)u_{n-1}-u_{n-2}$, if $n\geq3$.
\end{lemma}
\begin{proof} 
	Let $u_n=s_n-2$, where $n\geq1$. Then  $u_1=0$,  $u_2=1$  and $u_3=s_{3}-2=q-2=(q-2)u_{2}-u_{1}$. 
	
	For general cases corresponding to $n\ge4$, firstly, we have 
	\begin{eqnarray*}
		u_n&=& (q-1)s_{n-1}-(q-1)s_{n-2}+s_{n-3}-2\\
		&=& (q-1)(s_{n-1}-2)-(q-1)(s_{n-2}-2)+(s_{n-3}-2)\\
		&=& (q-1)u_{n-1}-(q-1)u_{n-2}+u_{n-3}.
	\end{eqnarray*}
	This also means, that $\{s_n\}$ and $\{u_n\}$ have the same ternary recurrence relation (with different initial values). 
	
	Secondly, we show, that  $\{u_n\}$ can be described by a binary recurrence relation too. (In contrast $\{s_n\}$  cannot.)
	Adding the equations $u_n=(q-2)u_{n-1}-u_{n-2}$ and $-u_{n-1}=-(q-2)u_{n-2}+u_{n-3}$, we obtain $u_n=(q-1)u_{n-1}-(q-1)u_{n-2}+u_{n-3}$. 
\end{proof}

The  first few terms of $\{u_i\}$ are $0$, $1$, $q-2$, $q^2-4q+3$, $q^3-6q^2+10q-4$, $q^4-8q^3+21q^2-20q+5$.

\begin{lemma}\label{lem:un}
	Both of the sub-sequences consisting of every second term of $\{F_i^{(a,b)}\}$ satisfy the relation 
	\begin{equation}\label{eq:xi}
	x_{i}=(ab+2)x_{i-2}-x_{i-4},\qquad (i\geq 4).
	\end{equation}
	Moreover,
	if $n\geq2$ then 
	\begin{equation}\label{eq:un}
	u_n=F_{2n-2}^{(1,q-4)}.
	\end{equation}
\end{lemma} 
\begin{proof} For the first few terms of $\{F_i^{(a,b)}\}$ the equation \eqref{eq:xi} is clearly true. We assume that for $i-1$ $(i\geq6)$ equation \eqref{eq:xi} also holds. Then  if $i$ is even, 
	\begin{eqnarray*}
		F_i^{(a,b)}&=&a F_{i-1}^{(a,b)}+F_{i-2}^{(a,b)}\\
		&=&a\left((ab+2)F_{i-3}^{(a,b)}-F_{i-5}^{(a,b)} \right)+\left((ab+2)F_{i-4}^{(a,b)}-F_{i-6}^{(a,b)} \right)\\
		&=& (ab+2)\left(aF_{i-3}^{(a,b)}+F_{i-4}^{(a,b)} \right) -\left(aF_{i-5}^{(a,b)}+F_{i-6}^{(a,b)} \right) \\
		&=& (ab+2)F_{i-2}^{(a,b)}-F_{i-4}^{(a,b)}.
	\end{eqnarray*}
	If $i$ is odd, the proof is the same. 
	For the case $a=1$ and $b=q-4$ we obtain the equation \eqref{eq:un}. 
\end{proof}

\begin{figure}[h!]
	\centering
	\includegraphics[width=0.8\linewidth]{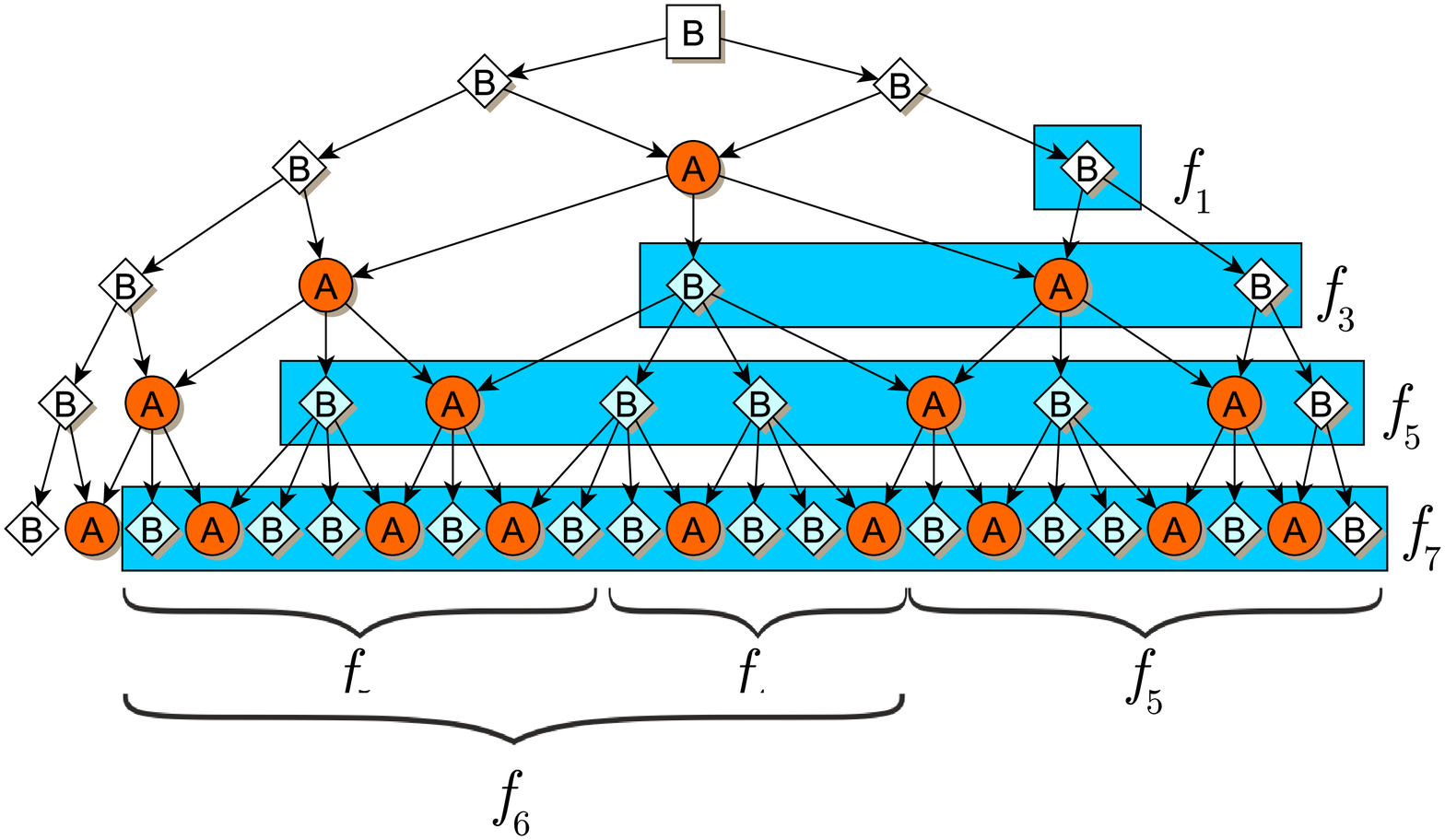}
	\caption{Pattern of \hpT45 up to row 5 and some Fibonacci words}
	\label{fig:Pascal_layer5_AB}
\end{figure}

Let $\{h_n^{[q]}\}_0^{\infty}$ be the sequence over $\{A,B\}$, where $h_n^{[q]}$  equals to the concatenations of the type of the vertices of row $n$ in \hpt4q from left to the right. Further, we call the elements of this the $\{4,q\}$-hyperbolic Pascal words (shortly $q$-hyperbolic Pascal words). For example in the case of $q=5$ (see Figure~\ref{fig:Pascal_layer5_AB}), we have
\begin{eqnarray*}
	h_0^{[5]}&=&B,\  h_1^{[5]}=BB,\  h_2^{[5]}=BAB,\  h_3^{[5]}=BABAB,
	\  h_4^{[5]}=BABABBABAB,\\
	h_5^{[5]}&=&BABABBABABBABBABABBABAB.
\end{eqnarray*}

Let us consider the bijection
\begin{equation}\label{eq:biject}
\phi: \{0,1\}\rightarrow \{A,B\},\qquad \phi(1)= A, \qquad \phi(0)= B.
\end{equation}

Let the words $u$ and $v$ be over $\{0,1\}$ and $\{A,B\}$, respectively. If $\phi(u)=v$, then we say that $u$ is equivalent to $v$ and we denote $u\equiv v$. 
For example from  Figure~\ref{fig:Pascal_layer5_AB} we have
\begin{align}
f_{1}=0&\equiv B=h_0^{[5]},&          0f_{1}=00&\equiv BB=h_1^{[5]}, \nonumber\\
f_{3}=01f_{1}=010&\equiv BAB=h_2^{[5]}, &	01f_{3}=01010&\equiv BABAB= h_3^{[5]}.\label{eq:01011}
\end{align}

Examining Figure~\ref{fig:Pascal_layer5_AB} we can recognise that every second Fibonacci word is almost equivalent to the patterns of the rows in \hpT45. (Compare the patterns of rows in Figure~\ref{fig:Pascal_46_AB5} and $f_{2n-3}^{[6]}$, $n=2,3,4$.) The following theorem gives the exact relationship between \hpt4q and  $\{4,q\}$-Fibonacci words.   

\begin{figure}[h!]
	\centering
	\includegraphics[width=0.65\linewidth]{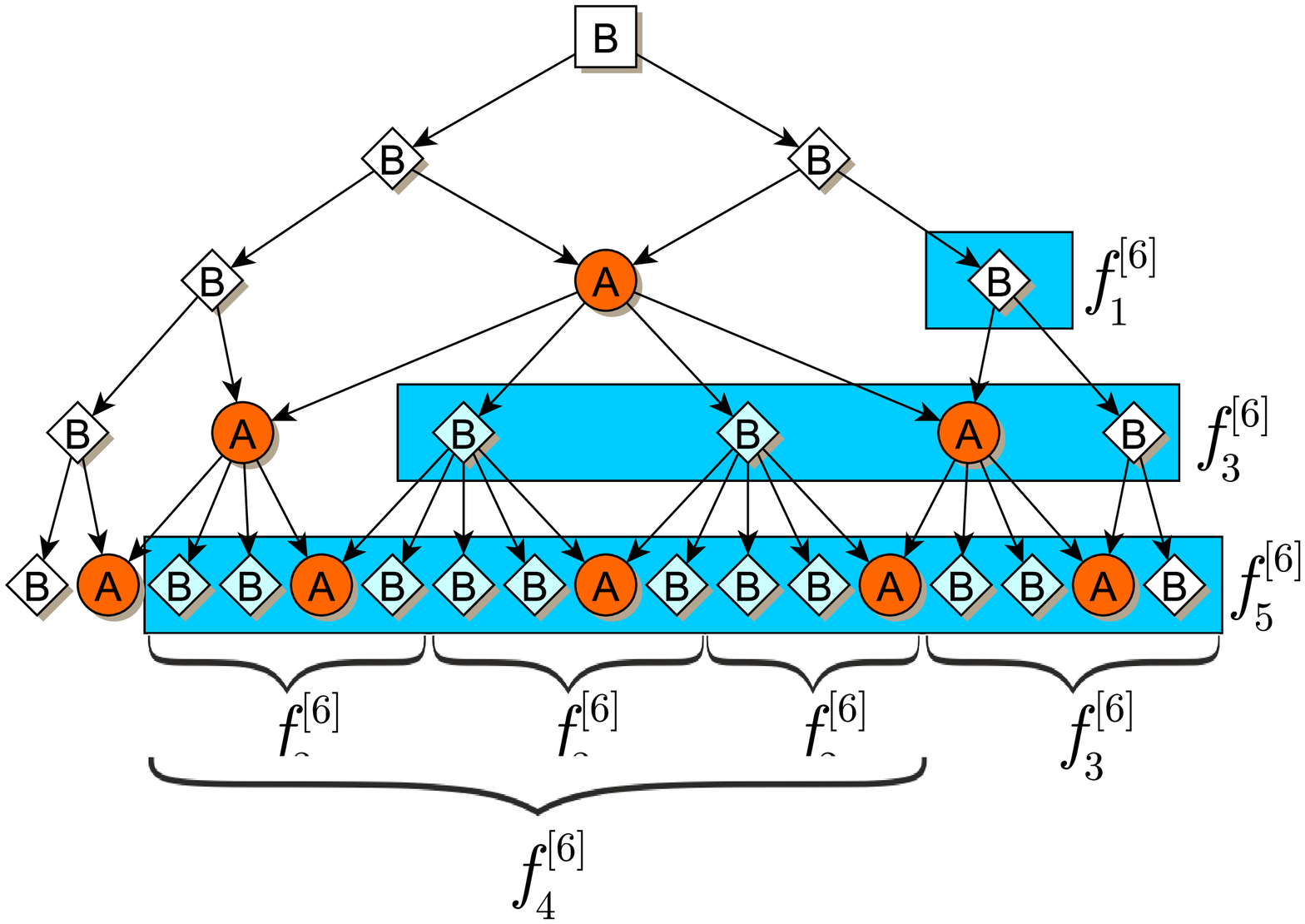}
	\caption{Pattern of \hpTT46 up to row 4 and some Fibonacci words}
	\label{fig:Pascal_46_AB5}
\end{figure}

\begin{theorem}\label{th:ab01}
	If $n\geq2$, then
	\begin{equation}\label{eq:fhq}
	01f_{2n-3}^{[q]}\equiv h_n^{[q]}
	\end{equation}
	and 
	$$ |f_{2n-3}^{[q]}|=F_{2n-2}^{(1,q-4)},$$
	where $1\equiv A$, $0\equiv B$ and $|h_n^{[q]}|=s_n$.
\end{theorem}

\begin{proof}
	If $n=2$, then 
	$01f_{1}^{[q]}=010\equiv BAB=h_2^{[q]}$. 
	For higher values of $n$, examining the growing method of the hyperbolic Pascal triangles  row by row based on  Figure~\ref{fig:Pascal growing_4q2},  we can recognise that except for the first two elements it can be described by the morphism 
	\begin{equation}
	\lambda:\{A,B\}\rightarrow\{A,B\}* \quad \lambda(A)= (B)^{q-4} A,\quad \lambda(B)=(B)^{q-4} AB.\label{eq:def4qAB}
	\end{equation}
	After comparing $\lambda$ with the $\{4,q\}$-Fibonacci morphism $\sigma^{[q]}$  between every second $f_{i}^{[q]}$ according to Theorem~\ref{th:morf}, we can recognize that the growing methods (see Figure~\ref{fig:Pascal growing_4q2}, \eqref{eq:def4q}  and \eqref{eq:def4qAB})  are the same. 
	This proves the equation \eqref{eq:fhq}, because the first two elements of all rows ($n\geq2$) in \hpt4q are $B$ and $A$.
	
	The second statement is a consequence of Lemma  \ref{lem:un}.
\end{proof}

\section{Some properties of $\{4,q\}$-Fibonacci words}

Presumably, the connection between the $\{4,q\}$-Fibonacci words and the hyperbolic Pascal pyramids can open new opportunities for examining the Fibonacci words. We show some properties of $\{4,q\}$-Fibonacci words in which we use these connections.    

Let a binary word $u$ be the concatenation  of the words $v$ and $w$, thus $u=vw$. If we delete $w$ from the end of $u$, we get $v$. Let us denote it by $v=u\ominus w$. 
In words, the sign $\ominus$ acts so, that  the word after the sign is deleted from the end of the word before the sign (if it is possible).  
For example
$f_{4}= f_{5}\ominus f_{3}=01001\xcancel{010}$,
$f_{6}= f_{5}f_{5}\ominus f_{3}=01001010\cdot 01001\xcancel{010}$ 
or 
$f_{4}^{[6]}=(f_{3}^{[6]})^{3}\ominus f_{5}^{[6]}= 0010\cdot 0010\cdot 001\xcancel{0}$.

\begin{theorem}\label{lem:fi2} 	All  $\{4,q\}$-Fibonacci words with $(k\geq2)$ can be given in terms of the previous two odd indexed ones, namely
	\begin{eqnarray*} 
		f_{2k}^{[q]}&=& \left( f_{2k-1}^{[q]}\right)^{q-3} \ominus f_{2k-3}^{[q]},   \\
		f_{2k+1}^{[q]}&=& \left(  \left( f_{2k-1}^{[q]}\right)^{q-3} \ominus f_{2k-3}^{[q]} \right)f_{2k-1}^{[q]}.
	\end{eqnarray*}
\end{theorem}
\begin{proof}
	Applying $f_{2k-1}^{[q]}= f_{2k-2}^{[q]}f_{2k-3}^{[q]}$ we can easily see that 
	$f_{2k-2}^{[q]}=f_{2k-1}^{[q]} \ominus f_{2k-3}^{[q]}$. 
	Furthermore, we can also see that  
	$f_{2k}^{[q]}=\left(f_{2k-1}^{[q]}\right)^{q-4}f_{2k-2}^{[q]}= \left(f_{2k-1}^{[q]}\right)^{q-4} f_{2k-1}^{[q]} \ominus f_{2k-3}^{[q]}=
	\left( f_{2k-1}^{[q]}\right)^{q-3} \ominus f_{2k-3}^{[q]}$.
	The second equation is the corollary of the first one.
\end{proof}

If $q$ tends to infinity, then the numbers of '0' in infinite $\{4,q\}$-Fibonacci words are relatively fast growing (see Table~\ref{tab:inffibon_word}). Now let us derive these ratios.

Let  $d_{i}^{[q]}$, $d_{i,0}^{[q]}$ and $d_{i,1}^{[q]}$ denote the numbers of  all, '0'  and '1' digits in the finite $\{4,q\}$-Fibonacci words, respectively.
Then, let the limit
$r_0^{[q]}=\lim_{i\to \infty} ({d_{i}^{[q]}}/{d_{i,0}^{[q]}})$ 
be  the inverse density of '0' digits in the infinite $\{4,q\}$-Fibonacci word.
Similarly, we denote the same density by
$r_1^{[q]}=\lim_{i\to \infty} ({d_{i}^{[q]}}/{d_{i,1}^{[q]}})$
in the case of '1' digits. 

\begin{theorem}\label{th:growing4q} The inverse density  of '0'  and '1' digits in the  infinite $\{4,q\}$-Fibonacci words are	
	\begin{eqnarray*}
		r_0^{[q]} &=& \frac{q-4+\sqrt{q(q-4)}}{2(q-4)},\\
		r_1^{[q]} &=& \frac{q-2+\sqrt{q(q-4)}}{2},
	\end{eqnarray*}
	where $q\geq5$. Moreover
	\begin{equation*}
	\lim_{q\to \infty} r_0^{[q]}=1 \quad \text{and} \quad \lim_{q\to \infty} r_1^{[q]}=\infty.
	\end{equation*}
\end{theorem}

\begin{proof}
	Firstly, let $i$ be odd and large enough, so that $i=2n-3$. As $01f_{2n-3}^{[q]}\equiv h_n^{[q]}$ from Theorem \ref{th:ab01}, we consider  the ratio $s_n/a_n$ from the hyperbolic Pascal triangle  instead of the corresponding ratio ${d_{2n-3}^{[q]}}/{d_{2n-3,0}^{[q]}}$. 
	Not only the sequence $\{s_n\}$ can be described by the ternary recurrence relation \eqref{recur1} but also the sequences $\{a_n\}$ and $\{b_n\}$ (more details in \cite{BNSz}). The solutions of the characteristic equations of their recurrence relations are positive real numbers. Moreover, it is well-known that the limit of $s_n/a_n$ is the density of the coefficients of the largest solutions (all solutions are positive), i.e. 
	$\alpha_s=-1/2+(q-2) \sqrt{q^2-4q}/(2q(q-4))$, $\alpha_a=(2-q)(1/2)+(q^2-4q+2)\sqrt{q^2-4q}/(2q(q-4))$ and
	$\alpha_b=(q-3)(1/2)+(1-q)\sqrt{q^2-4q}/(2q)$. 
	Thus,
	\begin{eqnarray*}
		\lim_{n\to \infty} \frac{d_{2n-3}^{[q]}}{d_{2n-3,0}^{[q]}} =
		\lim_{n\to \infty} \frac{s_{n}}{b_{n}} = \lim_{n\to \infty} \frac{\alpha_s}{\alpha_b}&=& \frac{q-4+\sqrt{q^2-4q}}{2(q-4)}, \\
		\lim_{n\to \infty} \frac{d_{2n-3}^{[q]}}{d_{2n-3,1}^{[q]}} =
		\lim_{n\to \infty} \frac{s_{n}}{a_{n}} = \lim_{n\to \infty} \frac{\alpha_s}{\alpha_a}&=&\frac{q-2+\sqrt{q^2-4q}}{2}.
	\end{eqnarray*}
	Secondly, let $i$ be even.  According to Theorem \ref{lem:fi2} all the even indexed $\{4,q\}$-Fibonacci words can be derived in terms of the previous two elements. We also obtain, that
	$d_{2k}^{[q]}=(q-3)d_{2k-1}^{[q]}-d_{2k-3}^{[q]}$, $d_{2k,0}^{[q]}=(q-3)d_{2k-1,0}^{[q]}-d_{2k-3,0}^{[q]}$ and 
	$d_{2k,1}^{[q]}=(q-3)d_{2k-1,1}^{[q]}-d_{2k-3,1}^{[q]}$.
	From it we have
	\begin{eqnarray*}
		\lim_{n\to \infty} \frac{d_{2k}^{[q]}}{d_{2k,0}^{[q]}} =
		\lim_{n\to \infty} \frac{(q-3)d_{2k-1}^{[q]}-d_{2k-3}^{[q]}}{(q-3)d_{2k-1,0}^{[q]}-d_{2k-3,0}^{[q]}} =\lim_{n\to \infty} \frac{d_{2k-1}^{[q]}}{d_{2k-1,0}^{[q]}}, 
	\end{eqnarray*}
	and  the case for digits '$1$' is similar.
	For the limits of $r_0^{[q]}$ and $r_1^{[q]}$, the statement is obviously true.
\end{proof}

Naturally, if $q=5$ the results of Theorem \ref{th:growing4q} give  the known $r_0^{[5]}=\varphi$ and $r_1^{[5]}=1+\varphi$ values, where $\varphi$ is the golden ratio.

\pagebreak[2]
Finally, here are some properties, which can directly be obtained from  the properties of \hpt4q:
\begin{itemize}
	\item The words $01f_{2n-3}^{[q]}$  ($n\geq2$) are palindromes.
	\item The subword $11$ never occurs in $\{4,q\}$-Fibonacci words.
	\item The subword $00\ldots0$ ($q-2$ digits 0) never occurs in words $f_{i}^{[q]}$.	
	\item The last two digits of finite $\{4,q\}$-Fibonacci words are alternately 01 and 10. 
	\item The infinite $\{4,q\}$-Fibonacci word has $n+1$ distinct subwords of length $n$, where $n\leq q-2$. In case $n=q-2$, they are  $100\ldots01$ with $q-4$ digits $0$ and the others are with only one digit $1$, in case  $n<q-2$ the subwords have at most one digit $1$.
\end{itemize}

\end{document}